\documentclass[11pt]{article}
\usepackage{amsmath,amssymb,amsthm}
\usepackage{booktabs} 
\usepackage{tikz}
\usetikzlibrary{calc}
\usepackage{float}
\usepackage{xcolor}
\usepackage{authblk}
\usepackage{autobreak}
\usepackage{enumerate}
\usepackage{hyperref}
\hypersetup{
	colorlinks,
	linkcolor={red!100!black},
	citecolor={green!60!black},
	urlcolor={blue!60!black}
}

\newtheorem{theorem}{Theorem}
\newtheorem{lemma}{Lemma}[section]
\newtheorem{proposition}{Proposition}[section]

\theoremstyle{definition}

\newtheorem{claim}{Claim}[section]
\newtheorem{conjecture}{Conjecture}

\newenvironment{subproof}[1][\proofname]{%
  \begin{proof}[#1]%
}{%
  \end{proof}%
}

\begin{document}

\title{New bounds for Ramsey numbers involving graphs with a center}

\author[,\,1]{Yanbo Zhang\thanks{Corresponding author.}}
\affil[1]{School of Mathematical Sciences\\ Hebei Normal University\\ Shijiazhuang 050024, China}
\author[2]{Yaojun Chen}
\affil[2]{School of Mathematics\\ Nanjing University\\ Nanjing 210093, China}

\date{}
\maketitle
\let\thefootnote\relax\footnotetext{\emph{Email addresses:} {\tt ybzhang@hebtu.edu.cn} (Yanbo Zhang), {\tt yaojunc@nju.edu.cn} (Yaojun Chen)}
\begin{quote}
{\bf Abstract:} Let $F_n$,  $W_n$, and $\widehat{K}_n$ be the graphs obtained by joining a vertex to $n$ independent edges, a cycle and a path of order $n-1$, respectively. In this paper, we give new bounds for the Ramsey numbers $R(F_n,F_m)$ and $R(W_n,W_n)$, which improve those due to Chen, Yu, and Zhao [EJC, 2021] and Mao, Wang, Magnant, and Schiermeyer [G\&C, 2022], respectively, and establish lower and upper bounds for $R(\widehat{K}_n,\widehat{K}_n)$. Moreover, we present a blow-up technique to establish some new lower bounds for the Ramsey numbers of wheels versus cliques.



\end{quote}

\section{Introduction}
Given two graphs $G_1$ and $G_2$, the Ramsey number $R(G_1, G_2)$ is defined as the smallest positive integer $N$ such that any red-blue edge-coloring of $K_N$ contains either a red $G_1$ as a subgraph or a blue $G_2$ as a subgraph. When $G_1=G_2$, we abbreviate $R(G_1, G_2)$ as $R(G_1)$. For the latest results on Ramsey numbers, refer to the surveys~\cite{Conlon2015,Radziszowski2024}.

This paper aims to improve the bounds on the Ramsey numbers for certain classes of graphs. We denote by $nK_2$, $P_n$, $C_n$, and $K_n$ a matching of $n$ edges, a path, a cycle, and a complete graph on $n$ vertices, respectively. We use $K_n-e$ to denote the graph obtained by removing an edge from $K_n$. The graph $K_1+G$ is formed by adding an extra vertex to the graph $G$, and connecting this extra vertex to every vertex in $G$. In this construction, the graphs $K_1+nK_2$, $K_1+C_n$, and $K_1+P_n$ are referred to as a fan, a wheel, and a kipas,  and are denoted by $F_n$, $W_{n+1}$, and $\widehat{K}_{n+1}$, respectively. Note that in many references, the  wheel  and the kipas on $n+1$ vertices are denoted by $W_n$ and $\widehat{K}_n$, respectively. In this paper, the subscripts of these two symbols represent their orders. The Ramsey numbers for fans, wheels, and kipas are widely studied.

For fans, Chen, Yu, and Zhao~\cite{Chen2021} established the following bounds for $R(F_n,F_m)$.

\begin{theorem}[Chen, Yu, and Zhao~\cite{Chen2021}] \label{chen-1}
Let $m \ge 4$. Then $R(F_n,F_m) \ge 3n+3m/2-5$ for $m < n \le 3m/2 - 3$ and $9n/2-5 \le R(F_n) \le 11n/2+6$.
\end{theorem}

Later, Dvo\v{r}\'ak and Metrebian~\cite{Dvorak2023} improved the upper bound for the diagonal case $R(F_n)$ to $31n/6+15$.

For wheels, Mao, Wang, Magnant, and Schiermeyer~\cite{Mao2022} established the following lower and upper bounds for $R(W_n)$.
\begin{theorem}[Mao, Wang, Magnant, and Schiermeyer~\cite{Mao2022}]  \label{Mao-1} For $n \ge 7$, when $n$ is even, $3n-3 \le R(W_n) \le 8n-10$; and when $n$ is odd, $2n-2 \le R(W_n) \le 6n-8$.
\end{theorem}

One of our main goals is to improve the lower bound for $R(F_n,F_m)$ and general lower and upper bounds for $R(W_n)$. The new bounds are presented in the following two theorems.

\begin{theorem}\label{thmfan}
Let $m \ge 4$. Then
  \begin{equation*}
    R(F_n, F_m) \ge  
    \begin{cases} 
    4n + m/2, & \text{for } m \le n \le 5m/4 - 1;\\
    2n + 3m - 2, & \text{for } 5m/4 - 1 < n \le 3m/2 - 2.
    \end{cases}
  \end{equation*}
  In particular, $R(F_n) \ge \left\lceil 9n/2 \right\rceil$ for $n \ge 4$.
\end{theorem}

By comparing the lower bounds in Theorems \ref{chen-1} and \ref{thmfan}, one can see that the new bound increases by $n-m+5$ when $m \le n \le 5m/4-1$, and by $3m/2-n+3$ when $5m/4-1 \le n \le 3m/2-3$. 


\begin{theorem}\label{thmwheel}
  Assume that $n\ge 8$. If $n$ is even, then \[3n-2\le R(W_n)\le \min\{6n-10, (16n+1978)/3\}.\] If $n$ is odd, then \[(5n-6+\operatorname{sgn}(n))/2\le R(W_n)\le 4n+\min\{(n-9)/2, 660\},\] where
  \[
    \operatorname{sgn}(n) = \begin{cases}
    1 & \text{if } n \equiv 3 \pmod{4}, \\
    -1 & \text{if } n \equiv 1 \pmod{4}.
    \end{cases}
    \]
\end{theorem}

The above theorem is applicable for $n \ge 8$, as the exact values of $R(W_n)$ are already known for $3 \le n \le 7$. For further details, refer to the survey~\cite{Radziszowski2024}. Clearly, Theorem~\ref{thmwheel} greatly narrows the gap between the lower and upper bounds in Theorem~\ref{Mao-1}.

The study of diagonal Ramsey numbers for fans and wheels has motivated us to establish general bounds for the diagonal Ramsey number of another well-known graph, the kipas. Another main task is to establish the bounds for $R(\widehat{K}_n)$.

\begin{theorem}\label{thmkipas}
  Assume that $n\ge 5$. We have \[(5n-6+\operatorname{sgn}'(n))/2\le R(\widehat{K}_n)\le 4n-6,\] where
  \[
    \operatorname{sgn}'(n) = \begin{cases}
    1 & \text{if } n \equiv 3 \pmod{4}, \\
    0 & \text{if } n \equiv 0 \pmod{2}, \\
    -1 & \text{if } n \equiv 1 \pmod{4}.
    \end{cases}
    \]
\end{theorem}

The final task of this paper is to  utilize the blow-up technique to establish the lower bound for some Ramsey numbers involving wheels. Before doing this, let us explain the concept of blow-up. 

Given two simple graphs $G$ and $H$, the blow-up graph $G[H]$ is formed by replacing each vertex of $G$ with a copy of $H$. In this blow-up graph, two vertices from different copies of $H$ are adjacent if and only if their corresponding original vertices in $G$ are adjacent. Referencing Figure~\ref{fig:blowup} will help in grasping this concept.

\begin{figure}[ht]
    \centering
\begin{tikzpicture}[scale=1, 
  vertex/.style={circle, draw, fill=white, inner sep=2.5pt},
  every label/.style={scale=0.8, black}]

  \node[vertex] (u1) at (-3, 4) {};
  \node[vertex] (u2) at (-5, 1) {};
  \node[vertex] (u3) at (-3, 1) {};
  \node[vertex] (u4) at (-1, 1) {};

  \draw (u1) -- (u2);
  \draw (u1) -- (u3);
  \draw (u1) -- (u4);

  \node[vertex] (u5) at (4, 4) {};
  \node[vertex] (u6) at (5, 4) {};
  \node[vertex] (u7) at (1, 1) {};
  \node[vertex] (u8) at (2, 1) {};
  \node[vertex] (u9) at (4, 1) {};
  \node[vertex] (u10) at (5, 1) {};
  \node[vertex] (u11) at (7, 1) {};
  \node[vertex] (u12) at (8, 1) {};

  \draw (u5) -- (u7);
  \draw (u5) -- (u8);
  \draw (u5) -- (u9);
  \draw (u5) -- (u10);
  \draw (u5) -- (u11);
  \draw (u5) -- (u12);
  \draw (u6) -- (u7);
  \draw (u6) -- (u8);
  \draw (u6) -- (u9);
  \draw (u6) -- (u10);
  \draw (u6) -- (u11);
  \draw (u6) -- (u12);
  \draw (u5) -- (u6);
  \draw (u7) -- (u8);
  \draw (u9) -- (u10);
  \draw (u11) -- (u12);
\end{tikzpicture}
\caption{The graph $K_{1,3}$ and its blow-up $K_{1,3}[K_2]$.}
\label{fig:blowup}
\end{figure}
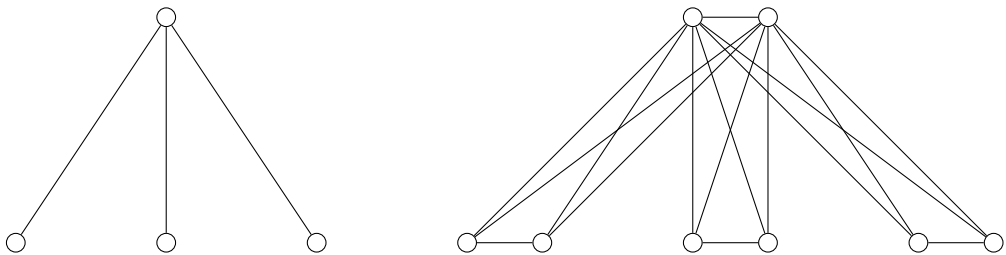

Using the blow-up technique, we can explicitly give the lower bound for $R(W_5, W_7)$, avoiding the help of computers.
\begin{theorem}\label{W5W7}
  $R(W_5, W_7) \ge 15$.
\end{theorem}

The lower and upper bounds for $R(W_5, W_7) = 15$ were obtained by Wesley~\cite{Wesley2023} and Van Overberghe (see in~\cite{Radziszowski2024}), respectively, followed by an independent proof from Lidick\'y, McKinley, and Pfender~\cite{Lidicky2024}. Wesley employed the SAT method to establish the lower bound, while the latter authors used tabu search for the same purpose.

This blow-up technique not only shows that $R(W_5, W_7) \ge 15$, but also yields many new lower bounds for $R(W_m, K_n)$. For instance, the Ramsey number $R(W_5, K_6) \ge 33$ was established by Shao, Wang, and Xiao~\cite{Shao2009} using a simulated annealing algorithm, and $R(W_6, K_6) \ge 34$ was provided by Goedgebeur and Van Overberghe~\cite{Goedgebeur2022} by checking block-circulant graphs. Our technique improves both of these lower bounds to $35$. The full theorem is stated as follows.

\begin{theorem}\label{WheelClique}
  For $5 \le m \le 6$ and all positive integers $n$, we have
  \[R(W_m, K_n) \ge 2R(K_3, K_n)-1\ \text{and}\ R(W_7, K_n)\ge 2R(K_4-e, K_n)-1\,.\]
  In particular, for $5 \le n \le 15$, the lower bounds of $R(W_m, K_n)$ are given in Table~\ref{W5W6Kn}; and for $5 \le n \le 10$, the lower bounds of $R(W_7, K_n)$ are given in Table~\ref{W7Kn}.
\end{theorem}

\begin{table}[ht]
  \centering
  \begin{tabular}{c ccccccccccc}
    \toprule
    The value of $n$ & 5 & 6 & 7 & 8 & 9 & 10 & 11 & 12 & 13 & 14 & 15 \\
    \midrule
    Lower bound & 27 & 35 & 45 & 55 & 71 & 79 & 93 & 105 & 119 & 133 & 147 \\
    \bottomrule
  \end{tabular}
  \caption{Lower bounds of $R(W_5, K_n)$ and $R(W_6, K_n)$ for $5 \le n \le 15$.}
  \label{W5W6Kn}
\end{table}

\begin{table}[ht]
  \centering
  \begin{tabular}{c ccccccc}
    \toprule
    The value of $n$ & 5 & 6 & 7 & 8 & 9 & 10 \\
    \midrule
    Lower bound & 31 & 41 & 55 & 71 & 81 & 97 \\
    \bottomrule
  \end{tabular}
  \caption{Lower bounds of $R(W_7, K_n)$ for $5 \le n \le 10$.}
  \label{W7Kn}
\end{table}

Somewhat surprisingly, by examining small Ramsey numbers~\cite{Radziszowski2024}, we find that for $1 \le n \le 5$, the lower bounds given for $R(W_5, K_n)$ are tight; and for $1 \le n \le 4$, the lower bounds provided for $R(W_7, K_n)$ are also tight. This leads us to boldly propose the following conjecture.

\begin{conjecture}
    For all positive integers $n$, we have \[R(W_5, K_n)=2R(K_3, K_n)-1\ \text{and}\ R(W_7, K_n)=2R(K_4-e, K_n)-1\,.\]
\end{conjecture}

The rest of this paper is organized as follows. Section~\ref{Secfans} focuses on the lower bound for Ramsey number of fans. The proofs of Theorems~\ref{thmwheel} and \ref{thmkipas} will be presented in Sections~\ref{Secwheels} and \ref{Seckipas}, respectively. 
Section~\ref{Secwheelscliques} is devoted to the proofs of Theorems~\ref{W5W7} and~\ref{WheelClique}.

We conclude this section by introducing additional notation as follows. For a graph $G$, let $V(G)$ and $E(G)$ denote the vertex set and edge set of $G$, respectively. The complement of $G$ is denoted by $\overline{G}$. The symbols $|G|$, $\delta(G)$, $\Delta(G)$, and $\kappa(G)$ represent the order (number of vertices), minimum degree, maximum degree, and connectivity of $G$, respectively. The set of neighbors of a vertex $u$ is denoted by $N(u)$, and $N[u]=N(u)\cup\{u\}$. The graph $G\cup H$ is the disjoint union of $G$ and $H$. For $U\subseteq V(G)$, $G-U$ denotes the graph obtained from $G$ by deleting the vertex set $U$ and all edges incident to $U$. In particular, when $U=\{u\}$, $G-U$ is simply written as $G-u$.

\section{Ramsey numbers of fans}\label{Secfans}

\begin{proof}[\textbf{\textit{Proof of Theorem~\ref{thmfan}}}]
We first construct the required extremal graph, as illustrated in Figure~\ref{fig:main}. In the complete graph $G$, there are five disjoint red complete subgraphs: $H_1$, $H_2$, $H_3$, $H_4$, and $K_{2n}$. The edges between the red $K_{2n}$ and its external vertices are all blue. The edges between the subgraphs $H_1 \cup H_4$ and $H_2 \cup H_3$ are all red. Additionally, all edges between $H_1$ and $H_4$ are blue, and all edges between $H_2$ and $H_3$ are also blue. The number of vertices in each $H_i$ depends on the specific case.

\vskip 5mm
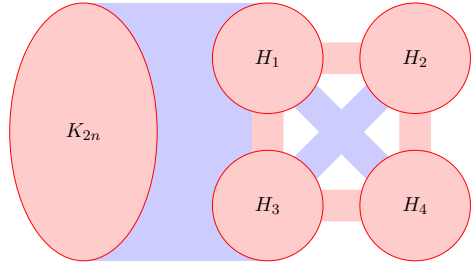
\begin{figure}[ht]
  \centering
  \scalebox{0.7}{
  \begin{tikzpicture}[scale=0.7, 
    vertex/.style={circle, draw, fill=white, inner sep=2pt},
    every label/.style={scale=0.8, black}]
  
    \fill[blue!20] (0, -3.5) rectangle (5,3.5);
  
    \draw[red!20, line width=6mm] (5, 2) -- (9, 2); 
    \draw[red!20, line width=6mm] (5, 2) -- (5, -2); 
    \draw[red!20, line width=6mm] (9, -2) -- (9, 2); 
    \draw[red!20, line width=6mm] (9, -2) -- (5, -2); 
    \draw[blue!20, line width=6mm] (5, 2) -- (9, -2); 
    \draw[blue!20, line width=6mm] (5, -2) -- (9, 2); 
  
    \draw[red, fill=red!20] (0, 0) ellipse (2cm and 3.5cm);
    \node at (0, 0) {$K_{2n}$};
  
    \draw[red, fill=red!20] (5, 2) circle (1.5cm);
    \node at (5, 2) {$H_1$};
  
    \draw[red, fill=red!20] (9, 2) circle (1.5cm);
    \node at (9, 2) {$H_2$};
  
    \draw[red, fill=red!20] (5, -2) circle (1.5cm);
    \node at (5, -2) {$H_3$};
  
    \draw[red, fill=red!20] (9, -2) circle (1.5cm);
    \node at (9, -2) {$H_4$};  
  \end{tikzpicture}}
\caption{The extremal graph $G$}
\label{fig:main}
\end{figure}

When $m \le n \le 5m/4-1$ and $m$ is even, let $|H_1| = m-1$, $|H_2| = 2n-3m/2+1$, $|H_3| = m/2$, and $|H_4| = m/2-1$.

When $m \le n \le 5m/4-1$ and $m$ is odd, let $|H_1| = m-1$, $|H_2| = 2n-3(m-1)/2$, $|H_3| = (m-1)/2$, and $|H_4| = (m-1)/2$.

When $5m/4-1 < n \le 3m/2-2$ and $m$ is even, let $|H_1| = m-1$, $|H_2| = m-1$, $|H_3| = m/2$, and $|H_4| = m/2-1$.

When $5m/4-1 < n \le 3m/2-2$ and $m$ is odd, let $|H_1| = m-1$, $|H_2| = m-1$, $|H_3| = (m-1)/2$, and $|H_4| = (m-1)/2$.

In each case, it can be verified that every vertex is adjacent to at most $2n-1$ red edges. Hence, the graph $G$ does not contain a red $F_n$ as a subgraph.

We proceed by contradiction, assuming that $G$ contains a blue $F_m$ as a subgraph. If the center vertex of this blue $F_m$ is in the red $K_{2n}$, then there must exist a matching consisting of $m$ blue edges in the subgraph $G-V(K_{2n})$. However, the vertex set $V(H_3) \cup V(H_4)$ forms a vertex cover for all blue edges in $G-V(K_{2n})$, and $|H_3| + |H_4| = m-1$. Thus, there cannot be a blue matching of $mK_2$ in $G-V(K_{2n})$. This contradiction implies that the center of the blue $F_m$ cannot be in the red $K_{2n}$.

Since each independent set of an $F_m$ contains at most $m$ vertices, the center vertex and at least $m$ non-center vertices of the blue $F_m$ must be in the subgraph $G-V(K_{2n})$. This implies that there is a vertex of $V(G)\setminus V(K_{2n})$ that is adjacent to at least $m$ blue edges in $G-V(K_{2n})$. However, it can be easily verified that in the subgraph $G-V(K_{2n})$, every vertex is adjacent to at most $m-1$ blue edges. This contradiction shows that no blue $F_m$ exists in $G$.

When $m \le n \le 5m/4-1$, the order of $G$ is given by $4n+\left\lceil m/2 \right\rceil -1$. When $5m/4-1<n \le 3m/2-2$, the order of $G$ becomes $2n+3m-3$. This completes the proof of the theorem.
\end{proof}  

\section{Ramsey numbers of wheels}\label{Secwheels} 
Before giving the proof of Theorem \ref{thmwheel}, as a preparation, we need introduce additional concepts and some known results.

A graph $G$ is said to be \emph{pancyclic} if it contains a cycle of every length from $3$ to $|G|$, and \emph{weakly pancyclic} if it contains a cycle of every length from the girth to the circumference. Moreover, we denote the circumference and girth of a graph $G$ by $c(G)$ and $g(G)$, respectively. Below we provide one condition for a graph to be pancyclic, two conditions for a graph to be weakly pancyclic, and one condition for the longest cycle in a graph. These conditions are very useful when searching for a cycle of a given length.
\begin{lemma}[Chen, Cheng, Ng, and Zhang~\cite{Chen2012}]\label{Chenlemma}
  $R(C_n,W_m)=3n-2$ for even $m$, $n \ge m-1 \ge 3$, and $(n,m) \neq (3,4)$.
\end{lemma}
\begin{lemma}[Bondy~\cite{Bondy1971}]\label{BondyLemma}
  Let $G$ be a graph of order $n$. If $\delta(G) \ge n/2$, then either $G$ is pancyclic or $n$ is even and $G = K_{n/2, n/2}$.
\end{lemma}

\begin{lemma}[Brandt, Faudree, and Goddard~\cite{Brandt1998}]\label{Brandt1lemma}
  Every nonbipartite graph $G$ of order $n$ with $\delta(G) \ge (n+2)/3$ is weakly pancyclic with $g(G)=3$ or $4$.
\end{lemma}

\begin{lemma}[Brandt, Faudree, and Goddard~\cite{Brandt1998}]\label{Brandt2lemma}
  Let $G$ be a 2-connected nonbipartite graph of order $n$ with minimum degree $\delta(G) \ge n/4+250$. Then $G$ is weakly pancyclic unless $G$ has odd girth $7$, in which case it has cycles of every length from $4$ up to its circumference except a $5$-cycle.
\end{lemma}

\begin{lemma}[Dirac~\cite{Dirac1952}]\label{Diraclemma}
  If $G$ is $2$-connected graph, then $c(G) \ge \min\{2\delta(G), |G|\}$.
\end{lemma}


\begin{lemma}[Rosta~\cite{Rosta1973}, Faudree and Schelp~\cite{Faudree1974}]\label{cycle}
\[R(C_m,C_n)=
\begin{cases}
2n-1 & \text{for } 3\le m\le n,\,m\text{ odd},\,(m,n)\neq(3,3),\\
n-1+m/2 & \text{for } 4\le m\le n,\,m,n\text{ even},\,(m,n)\neq(4,4),\\
\max\{n-1+m/2,2m-1\} & \text{for } 4\le m<n,\,m\text{ even and }n\text{ odd}.
\end{cases}\]
\end{lemma}

We present an upper bound for the Ramsey number of an even cycle versus a wheel of odd order, which is of independent interest.

\begin{lemma}\label{CycleWheelLemma}
  $R(C_{2n},W_{2m+1})\le \min\{\left\lfloor 9n/2 \right\rfloor, 4n+332\}$ for $n\ge m\ge 2$.
\end{lemma}

\begin{proof}
Consider a graph $G$ with $\min\{\left\lfloor 9n/2 \right\rfloor, 4n+332\}$ vertices. Let $v$ be a vertex of minimum degree in $G$.

First, assume that $d(v)\ge 3n-1$ in $\overline{G}$. Let $H$ be the subgraph induced by the non-neighbors of $v$. Hence, $|H|\ge 3n-1$. According to Lemma~\ref{cycle}, we have $R(C_{2n},C_{2m})=2n+m-1\le 3n-1$. Either $H$ contains $C_{2n}$ as a subgraph, or $\overline{H}$ contains $C_{2m}$ as a subgraph, which, together with vertex $v$, forms a $W_{2m+1}$ in $\overline{G}$. In either case, we obtain the desired graph. Thus, in the following, we only need to consider the case where $d(v)\le 3n-2$ in $\overline{G}$, that is, 
\[
\delta(G)\ge \min\{\left\lfloor 3n/2 \right\rfloor+1, n+333\}.
\]

If $\kappa(G)\le 1$, then there exists a vertex $u$ in the graph $G$ such that removing this vertex disconnects $G$. Note that if the graph $G$ is already disconnected, any vertex can be chosen as the vertex $u$. Since $\delta(G-u)\ge \min\{\left\lfloor 3n/2 \right\rfloor, n+332\}$, each connected component must have order at least \[\min\{\left\lfloor 3n/2 \right\rfloor+1, n+333\}.\]
  
If the graph $G-u$ has exactly two connected components, we denote the induced subgraphs of these components as $H_1$ and $H_2$. If $\delta(H_1)\le |H_1|-n-1$, then we select a vertex of minimum degree in $H_1$, denoted as $w$. Consider $m$ non-neighbors of $w$ in $H_1$ and $m$ non-neighbors of $w$ in $H_2$. These vertices must form a cycle $C_{2m}$ in $\overline{G}$, and when combined with vertex $w$, they create a $W_{2m+1}$ in the complement graph. Therefore, we may assume that $\delta(H_1)\ge |H_1|-n$. For the same reason, we can also assume that $\delta(H_2)\ge |H_2|-n$. By the pigeonhole principle, either $|H_1|\ge 2n$ or $|H_2|\ge 2n$. Without loss of generality, we assume the former holds. Thus, we have $\delta(H_1)\ge |H_1|/2$. According to Lemma~\ref{BondyLemma}, the graph $H_1$ contains $C_{2n}$ as a subgraph.
  
If the graph $G-u$ has at least three connected components, since each component contains at least $m$ vertices, we can select one vertex from one component and $m$ vertices from each of the other two components. The subgraph induced by these $2m+1$ vertices in $\overline{G}$ contains $W_{2m+1}$ as a subgraph.

If the graph $G$ is bipartite, then the larger part contains at least $2m+1$ vertices. Thus, $\overline{G}$ contains $W_{2m+1}$ as a subgraph.

Now we only need to consider the case where $\kappa(G) \ge 2$ and $G$ is not bipartite. If $n+333 \le \left\lfloor 3n/2 \right\rfloor +1$, it is easy to verify that
\[
\delta(G) \ge n+333 \ge (4n+332)/4+250 \ge |G|/4+250.
\]
According to Lemma~\ref{Brandt2lemma}, $G$ is weakly pancyclic (though it may not contain a 5-cycle). If $n+333 \ge \left\lfloor 3n/2 \right\rfloor +1$, it is easy to verify that
\[
\delta(G) \ge \left\lfloor 3n/2 \right\rfloor +1\ge (\left\lfloor 9n/2 \right\rfloor +2)/3 \ge (|G|+2)/3.
\]
By Lemma~\ref{Brandt1lemma}, $G$ is weakly pancyclic. The graph $G$ obviously contains $C_4$ as a subgraph. From Lemma~\ref{Diraclemma}, we can conclude that the length of the longest cycle in graph $G$ is at least $2\delta(G) \ge 2n$. Therefore, $G$ contains $C_{2n}$ as a subgraph. This completes the proof of the lemma.
\end{proof}

We now begin to prove Theorem \ref{thmwheel}. 
\vskip 2mm
First consider  the lower bound for $R(W_n)$. 

When $n$ is even, it is easy to verify that both the graph $3K_{n-1}$ and its complement do not contain $W_n$ as a subgraph. Therefore, we have $R(W_n) \ge 3n-2$.

When $n$ is odd, since $R(W_n) \ge R(\widehat{K}_n)$, the lower bound of $R(\widehat{K}_n)$ in Theorem~\ref{thmkipas} serves as the lower bound for $R(W_n)$, which will be proved in Section~\ref{Seckipas}.
\vskip 2mm
So we are left to show the upper bound for $R(W_n)$.
We divide the proof of the upper bound into three theorems.
\begin{theorem}\label{thmwheel-1}
  For even $n\ge 8$, we have $R(W_n) \le 6n-10$.
\end{theorem}


\begin{theorem}\label{thmwheel-2}
  For even $n\ge 8$, we have $R(W_n) \le (16n+1978)/3$.
\end{theorem}

\begin{theorem}\label{thmwheel-3}
  For odd $n\ge 9$, we have $R(W_n)\le 4n+\min\{(n-9)/2, 660\}$.
\end{theorem}

\begin{proof}[\textbf{\textit{Proof of Theorem~\ref{thmwheel-1}}}]
  Consider any graph $G$ with $6n-10$ vertices and any vertex $v$ in $G$. By the pigeonhole principle, either $d(v) \ge 3n-5$ in $G$ or $d(v) \ge 3n-5$ in $\overline{G}$. Without loss of generality, we assume the former holds. Let $H$ denote the subgraph induced by the neighborhood $N(v)$. Then, we have $|H| \ge 3n-5$. 

  By Lemma~\ref{Chenlemma}, we have $R(C_{n-1},W_n)=3n-5$ for $n \ge 5$. Therefore, either $H$ contains a subgraph $C_{n-1}$, which combined with vertex $v$ forms a $W_n$; or $\overline{H}$ contains $W_n$ as a subgraph. In either case, we obtain the desired $W_n$, thus completing the proof.
\end{proof}

\begin{proof}[\textbf{\textit{Proof of Theorem~\ref{thmwheel-2}}}]
This result is equivalent to proving $R(W_{2m+2}) \le (32m+2010)/3$ for $m \ge 3$.

Let $G$ be a graph with $\left\lfloor (32m+2010)/3 \right\rfloor$ vertices. We proceed by contradiction, assuming that both $G$ and its complement $\overline{G}$ do not contain $W_{2m+2}$ as a subgraph. 

We need the following two technical propositions. 
\begin{proposition}\label{threecliques}
  If a vertex $v$ satisfies $d(v) \ge (16m+1004)/3$, then in the neighborhood of $v$, there exist two vertices such that, after deleting these two vertices, the remaining neighbors of $v$ induce a subgraph that consists of three disjoint cliques. Furthermore, there are no edges between any two cliques; each clique contains at least $\left\lceil (4m+998)/3 \right\rceil$ vertices.
\end{proposition}

By the pigeonhole principle, for any vertex $v_1$, either $d(v_1)\ge (16m+1004)/3$ in $G$, or $d(v_1)\ge (16m+1004)/3$ in $\overline{G}$. By symmetry, assume the former holds. In the neighborhood of $v_1$, denote the three cliques obtained by Proposition~\ref{threecliques}  as $V_1$, $V_2$, and $V_3$.

\begin{proposition}\label{secvertex}
  If there exists a vertex $v_2 \in V_1 \cup V_2 \cup V_3$ such that $d(v_2) \ge (16m+1004)/3$, then there exists a monochromatic $W_{2m+2}$ in $G$.
\end{proposition}
To make the arguments easier to follow, we postpone the proofs of the two propositions until the end of this section.

If there exists a vertex of degree at least $(16m+1004)/3$ in $V_1 \cup V_2 \cup V_3$, then by Proposition~\ref{secvertex}, the proof is complete.

If there is no vertex of degree at least $(16m+1004)/3$ in $V_1 \cup V_2 \cup V_3$, then for every vertex $w$ in $V_1 \cup V_2 \cup V_3$, we have
\[
d(w) \ge (16m+1004)/3 \text{ in } \overline{G}.
\]
We select four vertices $w_1, w_2, w_3, w_4$ from $V_1 \cup V_2 \cup V_3$ such that $w_1$ is not adjacent to the other three vertices $w_2$, $w_3$, and $w_4$. According to Proposition~\ref{threecliques}, in the graph $\overline{G}$, after removing at most two vertices from the neighborhood of $w_1$, the induced subgraph of the remaining vertices forms three disjoint cliques. Let these cliques be denoted as $V_1'$, $V_2'$, and $V_3'$. Clearly, at least one of $w_2$, $w_3$, or $w_4$ belongs to $V_1' \cup V_2' \cup V_3'$, without loss of generality, assume $w_2 \in V_1' \cup V_2' \cup V_3'$. Then, the vertex $w_1$ corresponds to the vertex $v_1$ as defined before Proposition~\ref{secvertex}, and $w_2$ corresponds to the vertex $v_2$ in Proposition~\ref{secvertex}, while $V_1' \cup V_2' \cup V_3'$ corresponds to $V_1 \cup V_2 \cup V_3$ in Proposition~\ref{secvertex}. By Proposition~\ref{secvertex}, we can thus find a monochromatic $W_{2m+2}$, which completes the proof.
\end{proof}

\begin{proof}[\textbf{\textit{Proof of Theorem~\ref{thmwheel-3}}}]
  The upper bound is equivalent to proving $R(W_{2m+1}) \le 8m + \min\{m,664\}$ for $m \ge 4$.

  Consider a graph $G$ with $8m + \min\{m,664\}$ vertices and an arbitrary vertex $v$ in $G$. By the pigeonhole principle, either $d(v) \ge \min\{\left\lfloor 9m/2\right\rfloor,4m+332\}$ in $G$ or $d(v) \ge \min\{\left\lfloor 9m/2\right\rfloor,4m+332\}$ in $\overline{G}$. Without loss of generality, we assume the former holds. Let $H$ denote the subgraph induced by the neighborhood $N(v)$.

By Lemma~\ref{CycleWheelLemma}, we have $R(C_{2m},W_{2m+1}) \le \min\{\left\lfloor 9m/2\right\rfloor,4m+332\}$ for $m \ge 2$. Therefore, either $H$ contains a subgraph $C_{2m}$, which combined with vertex $v$ forms a $W_{2m+1}$; or $\overline{H}$ contains $W_{2m+1}$ as a subgraph. In either case, we obtain the desired $W_{2m+1}$, thus completing the proof.
\end{proof}
\vskip 2mm
\begin{proof}[\textbf{\textit{Proof of Proposition~\ref{threecliques}}}]
  Let $H$ be the subgraph induced by the neighborhood $N(v)$ of vertex $v$. Let $u$ be a vertex of minimum degree in $H$. If $\Delta(\overline{H}) \ge 4m+1$, then there are at least $4m+1$ vertices adjacent to $v$ but not adjacent to $u$. By Lemma~\ref{cycle}, the subgraph induced by these $4m+1$ vertices or its complement contains $C_{2m+1}$ as a subgraph. Thus, either $H$ contains a $C_{2m+1}$, which together with $v$ forms a $W_{2m+2}$, or $\overline{G}$ contains a $C_{2m+1}$, which together with $u$ forms a $W_{2m+2}$ in the complement. Both cases lead to a contradiction. Therefore, $\Delta(\overline{H}) \le 4m$, which implies $\delta(H) \ge |H|-4m-1$. It is easy to verify that
  \[
  \delta(H) \ge |H|/4 + 250 \ge (4m+1001)/3.
  \]
  If $H$ is a bipartite graph, then one part must contain at least $2m+2$ vertices. These vertices induce a complete graph in $\overline{G}$, which contains $W_{2m+2}$ as a subgraph, leading to a contradiction. Hence, $H$ is not bipartite. If $H$ is 2-connected, by Lemmas~\ref{Brandt2lemma} and~\ref{Diraclemma}, we can find a cycle $C_{2m+1}$ in $H$, which, together with vertex $v$, forms a $W_{2m+2}$, leading to a contradiction.

  If $H$ is not 2-connected, then there exists a vertex $u$ such that $H-u$ is disconnected. Note that if $H$ is already disconnected, removing any vertex $u$ will still leave $H-u$ disconnected. We have the following claim.

  \begin{claim}
    Either the subgraph $H-u$ has at least three connected components, or it has exactly two connected components, with at least one of them not being 2-connected.
  \end{claim}

  \begin{subproof}
    We proceed by contradiction, assuming that $H-u$ has exactly two connected components and both are 2-connected. This will lead to a contradiction.

    Let the two connected components be $H_1$ and $H_2$, with $|H_1| \ge |H_2|$. Thus, $|H_1| \ge (8m+501)/3$. Let $x$ be a vertex of minimum degree in $H_1$. If $\Delta(\overline{H_1}) \ge 2m+2$, then there must exist two vertices $y$ and $z$ in the non-neighbor set of $x$ such that $yz \in E(\overline{G})$. Otherwise, the non-neighbor set of $x$ induces a clique $K_{2m+2}$, which contains $W_{2m+2}$ as a subgraph, leading to a contradiction. Take $m+1$ vertices, including $y$ and $z$, from the non-neighbor set of $x$, and select $m$ vertices arbitrarily from $H_2$. These $2m+1$ vertices induce a $C_{2m+1}$ in $\overline{G}$, which, combined with vertex $x$, forms a $W_{2m+2}$ in $\overline{G}$, leading to another contradiction. Therefore, $\Delta(\overline{H_1}) \le 2m+1$, implying
    \[
    \delta(H_1) \ge |H_1|-2m-2.
    \]

    If $H_1$ is bipartite, then since $|H_1| \ge (8m+501)/3$, one part of $H_1$ must contain at least $m+2$ vertices. This part forms a complete graph $K_{m+2}$ in $\overline{G}$. Select $m$ vertices arbitrarily from $H_2$. Since there are no edges between these $m$ vertices and $K_{m+2}$, $\overline{G}$ contains a complete multipartite graph with $m+3$ parts, one of which contains $m$ vertices, and each of the others contains one vertex. This graph contains $W_{2m+2}$ as a subgraph. This contradiction shows that $H_1$ is not bipartite.

    If $|H_1| \ge (8m+1008)/3$, then $\delta(H_1) \ge |H_1|/4 + 250$. If $|H_1| \le (8m+1007)/3$, since $\delta(H_1) \ge \delta(H) - 1 \ge (4m+998)/3$, we still have $\delta(H_1) \ge |H_1|/4 + 250$. By Lemmas~\ref{Brandt2lemma} and~\ref{Diraclemma}, we can find a cycle $C_{2m+1}$ in $H_1$, which, together with vertex $v$, forms a $W_{2m+2}$, leading to a final contradiction.
  \end{subproof}

  Based on the above claim, if $H-u$ has exactly two connected components, then one of them contains a vertex $w$ such that $H-\{u, w\}$ has at least three connected components. If $H-u$ has at least three connected components, then removing any vertex $w$ from $H-u$ will leave $H-\{u, w\}$ with at least three connected components. Since $\delta(H-\{u, w\}) \ge (4m+995)/3$, each connected component contains at least
  \[
  \delta(H-\{u, w\}) + 1 \ge \left\lceil (4m+998)/3 \right\rceil
  \]
  vertices.

  If $H-\{u, w\}$ has at least four connected components, select $m$ vertices from each of two components, and one vertex from each of the other two components. These $2m+2$ vertices induce a $K_{1,1,m,m}$ in $\overline{G}$, meaning that $\overline{G}$ contains $W_{2m+2}$ as a subgraph, leading to a contradiction. Therefore, $H-\{u, w\}$ must have exactly three connected components.

  We claim that each connected component of $H-\{u, w\}$ is a clique. Otherwise, assume $w_1$ and $w_2$ belong to the same component, and $w_1w_2 \in E(\overline{G})$. Select $m$ vertices from each of the other two components. Together with $w_1$ and $w_2$, the induced subgraph in $\overline{G}$ contains $W_{2m+2}$ as a subgraph, leading to yet another contradiction. This completes the proof of the proposition.
\end{proof}

\begin{proof}[\textbf{\textit{Proof of Proposition~\ref{secvertex}}}]
  Without loss of generality, assume $v_2 \in V_1$. By Proposition~\ref{threecliques}, in the neighborhood of vertex $v_2$, after deleting at most two vertices, the remaining induced subgraph consists of three disjoint cliques. Denote these three cliques as $V_4$, $V_5$, and $V_6$. If one vertex is deleted from the neighborhood of $v_2$, denote it as $v_3$; if two vertices are deleted, denote them as $v_3$ and $v_4$. Let $V_0 = V_1 \setminus \{v_2, v_3, v_4\}$. Since $V_0 \subseteq N(v_2)$, and there are no edges between $V_4$, $V_5$, and $V_6$, it follows that $V_0$ is contained in one of these cliques. Without loss of generality, assume $V_0 \subseteq V_4$.
  
  Since every vertex in $V_5 \cup V_6$ is a neighbor of $v_2$, and every vertex in $V_2 \cup V_3$ is not a neighbor of $v_2$, the sets $V_2, V_3, V_5, V_6$ are four disjoint sets. Choose any vertex $v_5$ from $V_0$. It follows that there are no edges between $v_5$ and $V_2 \cup V_3 \cup V_5 \cup V_6$.
  
  If there is a vertex in $V_5$ that is not adjacent to some vertex in $V_2$, and another vertex in $V_5$ that is not adjacent to some vertex in $V_3$, for example, if $u_1u_2, u_3u_4 \in E(\overline{G})$ where $u_1, u_3 \in V_5$, $u_2 \in V_2$, and $u_4 \in V_3$, then the path $u_2u_1u_6u_3u_4$ exists in $\overline{G}$, where $u_6 \in V_6$. Since there are no edges between $V_2$ and $V_3$, a path of length $2m-3$ can be found between $u_2$ and $u_4$ in $\overline{G}[V_2 \cup V_3]$. This path, together with $u_2u_1u_6u_3u_4$, forms a cycle $C_{2m+1}$ in $\overline{G}$, which, combined with the central vertex $v_5$, forms a $W_{2m+2}$ in $\overline{G}$.

  From the previous analysis, either every vertex in $V_5$ is adjacent to every vertex in $V_2$, or every vertex in $V_5$, except for at most one, is adjacent to every vertex in $V_3$. By Proposition~\ref{threecliques}, $V_2$, $V_3$, and $V_5$ are cliques with at least $m+2$ vertices. Therefore, we obtain a clique with at least $2m+3$ vertices, which contains $W_{2m+2}$ as a subgraph. This completes the proof of the proposition.
\end{proof}

\section{Ramsey numbers of kipases}\label{Seckipas}
The proof of the upper bound relies on the following two classic lemmas.
\begin{lemma}[Gerencs\'er and Gy\'arf\'as~\cite{Gerencser1967}]\label{pathlemma}
  $R(P_n,P_m)=n+\left\lfloor m/2 \right\rfloor-1$.
\end{lemma}

\begin{lemma}[Dirac~\cite{Dirac1952}]\label{longestpath}
  Let $G$ be a connected graph of order $n\ge 3$. Then $G$ has a path of order $\min\{2\delta(G)+1,|G|\}$.
\end{lemma}

We first prove the upper bound of Theorem~\ref{thmkipas}.

\begin{lemma}
  $R(\widehat{K}_n)\le 4n-6$ for $n \ge 5$.
\end{lemma}

\begin{proof}
Let $G$ be a graph with $4n-6$ vertices, and let $v$ be any vertex in $G$. By the pigeonhole principle, either $d(v) \ge 2n-3$ in $G$ or $d(v) \ge 2n-3$ in $\overline{G}$. By symmetry, we assume the former holds. Let $H$ be the subgraph induced by the set of neighbors of $v$.

If $\delta(H) \le \left\lceil (n-1)/2 \right\rceil -1$, then $\Delta(\overline{H}) \ge 2n-\left\lceil (n-1)/2 \right\rceil -3$. By Lemma~\ref{pathlemma}, we have
\[
\Delta(\overline{H}) \ge R(P_{n-1}, P_{n-1})\,.
\]
Let $u$ be a vertex of minimum degree in $H$. Then, in the subgraph induced by the nonneighbors of $u$, either there exists a $P_{n-1}$, which combined with $v$ forms $\widehat{K}_n$, or in its complement $\overline{H}$, there exists a $P_{n-1}$, which combined with $u$ forms $\widehat{K}_n$. Hence, we now assume
\[
\delta(H) \ge \left\lceil (n-1)/2 \right\rceil\,.
\]

If $H$ contains a connected component with at least $n-1$ vertices, then by Lemma~\ref{longestpath}, the longest path in this component has at least $\min\{2\delta(H)+1, n-1\}=n-1$ vertices. This path, together with vertex $v$, forms a $\widehat{K}_n$. Therefore, we assume that each connected component of $H$ has at most $n-2$ vertices. Since $|H|\ge 2n-3$, $H$ must have at least three connected components. Moreover, since each component has at least $\delta(H)+1 \ge \left\lceil (n+1)/2 \right\rceil$ vertices, we can select one vertex, $\left\lfloor (n-1)/2 \right\rfloor$ vertices, and $\left\lceil (n-1)/2 \right\rceil$ vertices from three components, respectively. These $n$ vertices induce a subgraph in $\overline{H}$ that must contain $\widehat{K}_n$ as a subgraph. This completes the proof of the upper bound.
\end{proof}

Next, we proceed to prove the lower bound of Theorem~\ref{thmkipas}.
\begin{lemma}
  For $n\ge 5$, we have $R(\widehat{K}_n)\ge (5n-6+\operatorname{sgn}'(n))/2$, where
  \[
    \operatorname{sgn}'(n) = \begin{cases}
    1 & \text{if } n \equiv 3 \pmod{4}, \\
    0 & \text{if } n \equiv 0 \pmod{2}, \\
    -1 & \text{if } n \equiv 1 \pmod{4}.
    \end{cases}
    \]
\end{lemma}
\begin{proof}
  When $n$ is even, this lower bound is equivalent to \[R(\widehat{K}_{2m+2}) \ge 5m+2\,.\] We construct a graph on $5m+1$ vertices, denoted by $K_{2m+1} \cup K_{m,m,m}$. This graph is $2m$-regular, and thus does not contain $\widehat{K}_{2m+2}$ as a subgraph. In the complement of this graph, after removing an independent set $\overline{K}_{2m+1}$, the remaining subgraph consists of three disjoint copies of $K_m$. However, after removing any independent set from $\widehat{K}_{2m+2}$, the remaining graph either contains a vertex of degree $m$ or a connected component on $2m+1$ vertices. Neither of these cases appears in the graph $K_m$. Therefore, $\widehat{K}_{2m+2}$ is not contained in the complement of $K_{2m+1} \cup K_{m,m,m}$, and we conclude that $R(\widehat{K}_{2m+2}) \ge 5m+2$.

  When $n \equiv 1 \pmod{4}$, this lower bound is equivalent to \[R(\widehat{K}_{2m+1}) \ge 5m-1 \text{ for even }m\,.\] We construct the extremal graph $K_{2m} \cup K_{m,m-1,m-1}$, which serves as the required graph. The analysis follows similarly to the previous case. In this graph, each vertex has degree at most $2m-1$, so it does not contain $\widehat{K}_{2m+1}$ as a subgraph. In the complement of this graph, after removing an independent set $\overline{K}_{2m}$, the remaining subgraph consists of $K_m \cup K_{m-1} \cup K_{m-1}$. However, after removing any independent set from $\widehat{K}_{2m+1}$, the remaining graph either contains a vertex of degree $m$ or a connected component on $2m$ vertices. Neither of these cases occurs in $K_m \cup K_{m-1} \cup K_{m-1}$. Therefore, $\widehat{K}_{2m+1}$ is not contained in the complement of $K_{2m} \cup K_{m,m-1,m-1}$, yielding that $R(\widehat{K}_{2m+1}) \ge 5m-1$.
  
  It is worth noting that $K_{2m} \cup K_{m,m-1,m-1}$ is not the only extremal graph. Another example can be constructed by first taking the disjoint union of $K_m$ and $K_{m-1,m-1}$, and then adding $2m$ vertices, each of which is connected to the original $3m-2$ vertices. Using a similar analysis, this graph can also be shown to be extremal for $(\widehat{K}_{2m+1}, \widehat{K}_{2m+1})$.
  
  When $n \equiv 3 \pmod{4}$, this lower bound is equivalent to \[R(\widehat{K}_{2m+1}) \ge 5m \text{ for odd }m\,.\]

  The extremal graph we use is shown in Figure~\ref{fig:kipas}. It can be seen that this graph closely resembles Figure~\ref{fig:main}. We will discuss four cases based on the remainder of $m$ modulo $8$.
  
  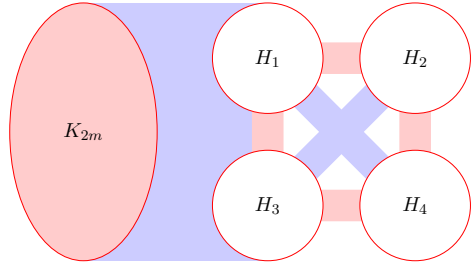
\begin{figure}[ht]
    \centering
    \scalebox{0.7}{
    \begin{tikzpicture}[scale=0.7, 
      vertex/.style={circle, draw, fill=white, inner sep=2pt},
      every label/.style={scale=0.8, black}]
    
      \fill[blue!20] (0, -3.5) rectangle (5,3.5);
    
      \draw[red!20, line width=6mm] (5, 2) -- (9, 2); 
      \draw[red!20, line width=6mm] (5, 2) -- (5, -2); 
      \draw[red!20, line width=6mm] (9, -2) -- (9, 2); 
      \draw[red!20, line width=6mm] (9, -2) -- (5, -2); 
      \draw[blue!20, line width=6mm] (5, 2) -- (9, -2); 
      \draw[blue!20, line width=6mm] (5, -2) -- (9, 2); 
    
      \draw[red, fill=red!20] (0, 0) ellipse (2cm and 3.5cm);
      \node at (0, 0) {$K_{2m}$};
    
      \draw[red, fill=white] (5, 2) circle (1.5cm);
      \node at (5, 2) {$H_1$};
    
      \draw[red, fill=white] (9, 2) circle (1.5cm);
      \node at (9, 2) {$H_2$};
    
      \draw[red, fill=white] (5, -2) circle (1.5cm);
      \node at (5, -2) {$H_3$};
    
      \draw[red, fill=white] (9, -2) circle (1.5cm);
      \node at (9, -2) {$H_4$};  
    \end{tikzpicture}}
  \caption{The extremal graph $G$}
  \label{fig:kipas}
  \end{figure}

If there exists an integer $k$ such that $m=8k+1$, let $|H_1|=6k+2$ and $|H_2|=|H_3|=|H_4|=6k$. The red induced subgraph of $H_1$ is a $(4k+1)$-regular graph, while the blue induced subgraph is a $2k$-regular graph. For both $H_2$ and $H_3$, their red induced subgraphs are $(4k-1)$-regular, and the blue induced subgraphs are $2k$-regular. The red induced subgraph of $H_4$ is $(4k+1)$-regular, and the blue induced subgraph is $(2k-2)$-regular.

If there exists an integer $k$ such that $m=8k+3$, let $|H_1|=|H_2|=|H_3|=|H_4|=6k+2$. The construction of the four graphs is identical: for each graph $H_i$, the red induced subgraph is $(4k+1)$-regular, and the blue induced subgraph is $2k$-regular.

If there exists an integer $k$ such that $m=8k+5$, let $|H_1|=|H_2|=|H_3|=6k+4$ and $|H_4|=6k+2$. The red induced subgraph of $H_1$ is $(4k+1)$-regular, and the blue induced subgraph is $(2k+2)$-regular. For both $H_2$ and $H_3$, their red induced subgraphs are $(4k+3)$-regular, and the blue induced subgraphs are $2k$-regular. The red induced subgraph of $H_4$ is $(4k+1)$-regular, and the blue induced subgraph is $2k$-regular.

If there exists an integer $k$ such that $m=8k+7$, let $|H_1|=|H_2|=6k+6$ and $|H_3|=|H_4|=6k+4$. The red induced subgraphs of both $H_1$ and $H_2$ are $(4k+3)$-regular, and the blue induced subgraphs are $(2k+2)$-regular. Similarly, for $H_3$ and $H_4$, the red induced subgraphs are $(4k+3)$-regular, and the blue induced subgraphs are $2k$-regular.

In all four cases, it can be verified that the red induced subgraph of $G$ is a $(2m-1)$-regular graph. Therefore, $G$ contains no red $\widehat{K}_{2m+1}$. The blue induced subgraph of $G-V(K_{2m})$ is an $(m-1)$-regular graph, with no blue connected component on $2m$ vertices. Hence, if $G$ contains a blue $\widehat{K}_{2m+1}$, its center must not lie in either $V(K_{2m})$ or $V(G)\setminus V(K_{2m})$. This analysis parallels the reasoning in Theorem~\ref{thmfan}. Therefore, $G$ also contains no blue $\widehat{K}_{2m+1}$.

In each case, it can be verified that $|G|=5m-1$. Thus, $R(\widehat{K}_{2m+1})\ge 5m$ for odd $m$.
\end{proof}

\section{Small wheel-wheel and wheel-clique Ramsey numbers}\label{Secwheelscliques}

The proofs of Theorems~\ref{W5W7} and~\ref{WheelClique} require the following observation.
\begin{lemma}
If a graph $G$ is triangle-free, then the blow-up graph $G[K_2]$ contains neither $W_5$ nor $W_6$ as subgraphs.
\end{lemma}
\begin{proof}
 If a graph $G$ is triangle-free, then for any vertex $u$ in $G[K_2]$, the subgraph induced by its closed neighborhood $N[u]$ cannot be 3-connected on at least five vertices (See Figure~\ref{fig:blowup}). However, both $W_5$ and $W_6$ are 3-connected graphs. Consequently, the graph $G[K_2]$ does not contain $W_5$ or $W_6$ as subgraphs. 
\end{proof}

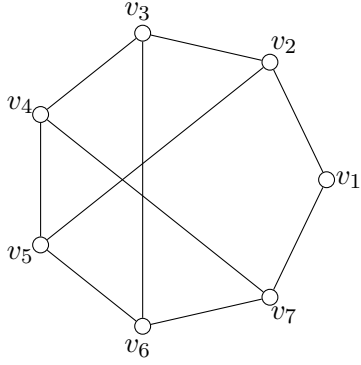
\begin{figure}[ht]
  \centering
  \begin{tikzpicture}[scale=1, every node/.style={inner sep=1pt, minimum size=6pt}]
      \foreach \i in {1, 2, ..., 7} {
          \node[circle, draw] (v\i) at ({360/7 * (\i - 1)}:2) {};
          \node[anchor=center] at ({360/7 * (\i - 1)}:2.3) {$v_\i$};
      }
    
      \draw (v1) -- (v2);
      \draw (v2) -- (v3);
      \draw (v3) -- (v4);
      \draw (v4) -- (v5);
      \draw (v5) -- (v6);
      \draw (v6) -- (v7);
      \draw (v7) -- (v1);
      \draw (v2) -- (v5);
      \draw (v3) -- (v6);
      \draw (v4) -- (v7);
    \end{tikzpicture}
  \caption{A cycle $C_7$ with three chords.}
  \label{fig:C7}
\end{figure}

\begin{proof}[\textbf{\textit{Proof of Theorem~\ref{W5W7}}}]
  To prove that $R(W_5, W_7) \ge 15$, consider the graph $G$ shown in Figure~\ref{fig:C7}, which is a 7-cycle $v_1v_2v_3v_4v_5v_6v_7v_1$ with three chords: $v_2v_5$, $v_3v_6$, and $v_4v_7$. Since this graph contains no triangles, $G[K_2]$ cannot contain $W_5$ as a subgraph. We only need to verify that $\overline{G[K_2]}$ does not contain $W_7$ as a subgraph, which would then establish that $R(W_5, W_7) \ge 15$. 
  
  If a vertex in $\overline{G[K_2]}$ is obtained by blowing-up one of the vertices from $\{v_3,v_4,v_5,v_6\}$, then the subgraph induced by its neighborhood is the graph $K_{2,4}$, which does not contain $C_6$ as a subgraph. If a vertex in $\overline{G[K_2]}$ is obtained by blowing-up one of the vertices from $\{v_1,v_2,v_7\}$, then in the subgraph induced by its neighborhood, each connected component has at most four vertices, and thus it does not contain $C_6$ as a subgraph. Therefore, there is no $W_7$ in $\overline{G[K_2]}$.
\end{proof}

Theorem~\ref{WheelClique} is derived from the following two lemmas and two tables.

\begin{lemma}
    For all positive integers $n$, we have \[R(W_5, K_n) \ge 2R(K_3, K_n)-1\ \text{and}\ R(W_6, K_n) \ge 2R(K_3, K_n)-1\,.\]
\end{lemma}

\begin{proof}
    Consider the Ramsey graph $G$ for $(K_3, K_n)$, which has $R(K_3, K_n)-1$ vertices and does not contain $K_3$ as a subgraph, while its complement does not contain $K_n$ as a subgraph. Thus, $G[K_2]$ does not contain $W_5$ or $W_6$ as subgraphs. If the complement $\overline{G[K_2]}$ contains $K_n$ as a subgraph, it is easy to see that $\overline{G}$ would also contain $K_n$ as a subgraph, which leads to a contradiction. Therefore, $R(W_m, K_n) \ge 2|G|+1=2R(K_3, K_n)-1$ for $m=5,6$.
\end{proof}

\begin{lemma}
    For all positive integers $n$, we have \[R(W_7, K_n) \ge 2R(K_4-e, K_n)-1\,.\]
\end{lemma}

\begin{proof}
    Consider the Ramsey graph $G$ for $(K_4-e, K_n)$, which has $R(K_4-e, K_n)-1$ vertices. This graph does not contain $K_4-e$ as a subgraph, and its complement does not contain $K_n$ as a subgraph.
    
    In the blow-up graph $G[K_2]$, for any vertex $u$, the subgraph induced by the closed neighborhood $N[u]$ cannot be 3-connected with at least seven vertices. If it were, then $G$ would contain a subgraph isomorphic to $K_4-e$, a contradiction. Since $W_7$ is a 3-connected graph on seven vertices, this implies that $G[K_2]$ does not contain $W_7$ as a subgraph.

    Moreover, if $\overline{G[K_2]}$ contains $K_n$ as a subgraph, then $\overline{G}$ would also contain $K_n$ as a subgraph, a contradiction. Therefore, we conclude that $R(W_7, K_n) \ge 2|G| + 1 = 2R(K_4-e, K_n)-1$.
\end{proof}

The exact values or lower bounds of the following small Ramsey numbers have been established by different researchers over the past seventy years. The contributors and references for these values can be found in Radziszowski's dynamic survey~\cite{Radziszowski2024}. The numbers highlighted in red represent the exact values of these Ramsey numbers.
\vskip 2mm

\begin{table}[ht]
  \centering
  \label{K3Kn}
  \begin{tabular}{c cccccccccccccc}
    \toprule
    The value of $n$ & 3 & 4 & 5 & 6 & 7 & 8 & 9 & 10 & 11 & 12 & 13 & 14 & 15 \\
    \midrule
    Lower bound & \textcolor{red}{6} & \textcolor{red}{9} & \textcolor{red}{14} & \textcolor{red}{18} & \textcolor{red}{23} & \textcolor{red}{28} & \textcolor{red}{36} & 40 & 47 & 53 & 60 & 67 & 74 \\
    \bottomrule
  \end{tabular}
  \caption{Lower bounds of $R(K_3, K_n)$ for $3 \le n \le 15$.}
\end{table}

\begin{table}[ht]
  \centering
  \label{K4-eKn}
  \begin{tabular}{c ccccccccc}
    \toprule
    The value of $n$ & 3 & 4 & 5 & 6 & 7 & 8 & 9 & 10 \\
    \midrule
    Lower bound & \textcolor{red}{7} & \textcolor{red}{11} & \textcolor{red}{16} & \textcolor{red}{21} & 28 & 36 & 41 & 49 \\
    \bottomrule
  \end{tabular}
  \caption{Lower bounds of $R(K_4-e, K_n)$ for $3 \le n \le 10$.}
\end{table}

The two tables above lead to the corresponding tables in Theorem~\ref{WheelClique}. This completes the proof of the theorem.

\subsection*{Acknowledgements}

Y. Chen was partially supported by National Key R\&D Program of China under grant number 2024YFA1013900 and NSFC under grant number 12471327.

\subsection*{Data Availability Statement}

No data was used or generated in this research.

\subsection*{Conflict of interest}

The authors declare that they have no known competing financial interests or personal relationships that could have appeared to influence the work reported in this paper.

\paragraph{Note added.}This paper was submitted for publication on Feb.\ 3, 2025. We later learned that DeBiasio and Wimbish independently obtained bounds for the Ramsey numbers of wheels in arXiv:2604.11937.

\end{document}